\documentclass{article}
\usepackage{amsthm}
\usepackage{amsmath}
\usepackage{amssymb}

\title{Abundance of fast growth of the number of periodic points\\
 in 2-dimensional area-preserving dynamics}
\author{Masayuki Asaoka}

\def\RR{\mathbb{R}}
\def\CC{\mathbb{C}}
\def\ZZ{\mathbb{Z}}
\def\TT{\mathbb{T}}

\def\cU{{\mathcal U}}
\def\cD{{\mathcal D}}
\def\cV{{\mathcal V}}

\def\cR{{\mathcal R}}

\def\ND{{\mathcal{ND}}}

\def\ra{{\rightarrow}}

\def\vphi{\varphi}
\def\del{\partial}

\DeclareMathOperator{\Diff}{Diff}
\DeclareMathOperator{\Ham}{Ham}
\DeclareMathOperator{\cH}{Hol}

\DeclareMathOperator{\Per}{Per}
\DeclareMathOperator{\Img}{Im}
\DeclareMathOperator{\Int}{Int}

\DeclareMathOperator{\supp}{supp}
\DeclareMathOperator{\tr}{tr}

\newcommand\cl[1]{{\overline{#1}}}

\theoremstyle{plain}
\newtheorem{thm}{Theorem}[section]
\newtheorem{prop}[thm]{Proposition}
\newtheorem{lemma}[thm]{Lemma}

\theoremstyle{definition}

\theoremstyle{remark}

\begin{document}
\maketitle

\begin{abstract}
We prove that there exists an open subset of
 the set of real-analytic Hamiltonian diffeomorphisms of a closed surface
 in which diffeomorphisms
 exhibiting fast growth of the number of periodic points are dense.
We also prove that there exists an open subset
 of the set of smooth area-preserving diffeomorphisms of a closed surface
 in which typical diffeomorphisms exhibit
 fast growth of the number of periodic points. 
\end{abstract}

\section{Introduction}
\label{sec:introduction}
In this article, we give an application of the KAM theory to
 a classical problem on the growth rate of the number of periodic points
 of a dynamical system.
We say that a continuous self-map $f$ on a topological space $X$ satisfies
 the condition (Exp) if the number of periodic points of $f$ grows
 at most exponentially, {\it i.e.},
 there exists $\lambda>1$ such that
\begin{equation*}
 \limsup_{n \ra \infty}
 \frac{1}{\lambda^n}\#\{x \in X \mid f^n(x)=x\}=0,
\end{equation*}
 where $\# S$ denotes the cardinality of a set $S$.
In 1965, Artin and Mazur \cite{AM} proved that $C^r$ maps satisfying
 (Exp) are dense in the space of $C^r$-maps on a closed manifold
 (see also \cite{Ka99}).
For one-dimensional dynamics,
 Theorem B' in \cite{MMS}
 by Martens, de Melo, and van Strien implies that
 a $C^r$ map on a compact interval with $r \geq 3$
 satisfies (Exp) if it does not admit any $(r-1)$-flat critical points
 (see also \cite[Corollary 1]{KK}).
In particular,
 the condition (Exp) is open and dense
 in the space of $C^r$ maps on a compact interval
 for any $3 \leq r \leq \infty$.
If we allow the existence of flat critical points,
 super-exponential growth is possible for interval maps.
For any given sequence of positive integers and any $1 \leq r<\infty$,
 Kaloshin and Kozlovski \cite{KK} constructed
 an example of $C^r$ unimodal map on a compact interval
 whose number of periodic points grows faster than the sequence.
\footnote{
 In the appendix of this paper, we give a simple example
 for $C^\infty$ case.}

For higher-dimensional dynamics,
 it turned out that super-exponential growth is abundant.
In \cite{Ka}, Kaloshin proved that super-exponential growth is generic
 in the Newhouse domain for $C^r$ surface diffeomorphisms
 with $2 \leq r \leq \infty$\footnote{
Kaloshin stated his result for finite $r$.
However, his proof works for the $C^\infty$ case.}.
Kaloshin and Saprykina \cite{KS}, Bonatti, D\'iaz, and Fisher \cite{BDF},
 and the author, Shinohara, and Turaev \cite{AST}
 also gave open sets of maps in which generic maps exhibit
 super-exponential growth in higher dimensions or in other settings.

It is natural to ask whether any real-analytic map
 satisfies the condition (Exp) or not.
Any real-analytic map on a compact interval satisfies (Exp)
 since it has no flat-critical point.
In higher dimension, as far as the author's knowledge,
 there was no known real-analytic example
 which exhibits super-exponential growth
 of the number of periodic points.\footnote{
A related example is given by Kozlovski \cite{Ko}.}
The first main result of this paper
 is the density of maps with super-exponential growth
 in an open subset of the set of real-analytic Hamiltonian diffeomorphisms
 on the two-dimensional torus.
Let $\TT^2$ be the two-dimensional torus $\RR^2/\ZZ^2$.
A diffeomorphism of $\TT^2$ is called a {\it Hamiltonian diffeomorphism}
 if it is a composition of finite number of the time-one maps
 of the Hamiltonian flows associated with
 real-analytic time-dependent Hamiltonian functions and 
 the standard volume form $dx \wedge dy$ on $\TT^2$.
Let $\Ham^\omega(\TT^2)$ be
 the space of real-analytic Hamiltonian diffeomorphisms
 of $\TT^2$ with the $C^\omega$-topology
 (see Section \ref{subsec:topology}
 for the definition of the $C^\omega$ topology).
For a diffeomorphism $f$ of a surface and $n \geq 1$,
 let $\Per_h(f,n)$ ({\it resp.,} $\Per_e(f,n)$) be the
 the set of hyperbolic ({\it resp.}, elliptic) periodic points
 of $f$ whose least period is $n$.
We call a periodic point {\it non-degenerate} if
 it is hyperbolic or elliptic.
\begin{thm}
\label{thm:main analytic}
There exists an open subset $\cU_*$ of $\Ham^\omega(\TT^2)$
 such that maps satisfying the following conditions
 are dense in $\cU_*$
  for any given sequence $(\gamma_n)_{n \geq 1}$ of positive integers:
\begin{enumerate}
\item All periodic point of $f$ are non-degenerate.
\item
$\displaystyle
\limsup_{n \ra \infty}
\frac{\min\{\#\Per_h(f,n),\#\Per_e(f,n)\}}{\gamma_n}=\infty.$
\end{enumerate}
\end{thm}
For $\gamma_n=2^{n^2}$,
 any map in the above dense set does not satisfy (Exp).
Remark that genericity is almost non-sense 
 in the space $\Ham^\omega(\TT^2)$ with our $C^\omega$ topology
 since it does not satisfies the Baire property.
In particualr, a residual subset of $\Ham^\omega(\TT^2)$
 may be empty in general.

We also apply the idea of the proof of Theorem \ref{thm:main analytic}
 to show that fast growth of the number of periodic points
 is $C^\infty$-typical in the sense of Kolmogorov-Arnold
 in an open subset of the set of $C^\infty$ area-preserving diffeomorphisms
 of a closed surface.
Let $M$ be a $C^r$ manifold with $1 \leq r \leq \infty$.
We denote the space of $C^r$ diffeomorphisms with the $C^r$-topology
 by $\Diff^r(M)$.
For a finite-dimensional closed unit disk $\Delta$
 and a subset $\cD$ of $\Diff^r(M)$,
 a family $(f_\mu)_{\mu \in \Delta}$
 is called {\it a $C^r$ family in $\cD$ with parameter set $\Delta$}
 if each $f_\mu$ is an element of $\cD$
 and the map $F(x,\mu)=f_\mu(x)$
 is a $C^r$ map from $M \times \Delta$ to $M$.
The identification of the family $(f_\mu)_{\mu \in \Delta}$
 with the $C^r$ map $F$ induces the $C^r$-topology to
 the set of $C^r$ family in $\cD$ with parameter set $\Delta$.
We say that a {\it $C^r$-typical} map in $\cD$ satisfies a given property
 if for any finite dimensional closed unit disk $\Delta$
 and any generic $C^r$ family $(f_\mu)_{\mu \in \Delta}$
 of maps in $\cD$,
 the set of parameters $\mu \in \Delta$
 for which the map $f_\mu$ satisfies the property
 contains a Lebesgue full measure subset of $\Delta$.
See \cite[Section 3.1]{IL} for more information on typicality
 and related concepts.
In many cases, it is harder to show that a dynamical property is typical
 than to show that it is generic.
For example, Newhouse \cite{Ne}
 proved that finiteness of periodic attractors
 is not generic for $C^r$ diffeomorphisms on a closed surface
 with $r \geq 2$ in 1970's.
But, $C^r$-typicality of finiteness of periodic attractors
 had been a long-standing open problems
 (see {\it e.g.}, \cite[Section 2.7]{Pa})
 until Berger \cite{Be} recently proved that
 this property is not $C^r$ typical
 for $d$-dimensional diffeomorphisms with $d \geq 3$
 and $ d \leq r \leq \infty$.
Typicality of the condition (Exp) has been also an open problem
 (see {\it e.g.}, \cite[Problems 1992-13, 1994-47]{Ar}).
The second main result of this paper
 shows that super-exponential growth is
 $C^\infty$-typical in an open subset
 of the set of $C^\infty$ {\it area-preserving}
 diffeomorphisms on a closed surface.
Let $\Sigma$ be a $C^\infty$ two-dimensional closed orientable manifold
 and $\Omega$ a $C^\infty$ volume form on $\Sigma$.
We denote the set of $C^\infty$ diffeomorphisms of $\Sigma$ which preserve
 the volume $\Omega$ by $\Diff^\infty_\Omega(\Sigma)$.
The set $\Diff^\infty_\Omega(\Sigma)$
 is endowed with the $C^\infty$-topology.
\begin{thm}
\label{thm:main smooth} 
There exists an open subset $\cU$
 of $\Diff^\infty_\Omega(\Sigma)$
 such that fast growth of the number of periodic points is $C^\infty$-typical
 in $\cU$.
More precisely,
 for any given sequence $(\gamma_n)_{n \geq 1}$ of positive integers,
 a $C^\infty$-typical map $f$ in $\cU$ satisfies
\begin{equation*}
 \limsup_{n \ra \infty}
 \frac{\min\{\#\Per_h(f,n),\#\Per_e(f,n)\}}{\gamma_n}=\infty.
\footnote{It seems possible
 to prove that $C^r$-typical
 diffeomorphisms in $\Diff^\infty_\Omega(\Sigma)$ has
 no degenerate periodic points
 by an argument analogous to the proof of $C^r$-typicality
 of Kupka-Smale diffeomorphisms in \cite{Ka97}.
}
\end{equation*}
\end{thm}
After the first version of this article was submitted to arXiv,
 Berger \cite{Be-pre} proved that the property (Exp) is not $C^r$-typical
 in $\Diff^r(M)$ for $1 \leq r <\infty$ and $\dim M \geq 3$.
Remark that Kaloshin and Hunt proved a result
 on the growth rate of the number of periodic points
 which contrasts with the above results on typical fast growth.
In \cite{KH1,KH2,KH3}, they proved that
 at most stretched exponential growth of the number of periodic points
 is {\it $C^r$-prevalent}
 in the sense of Hunt-Sauer-Yorke (see \cite{HSY,OY})
 for $1 \leq r \leq \infty$.

\paragraph{Acknowledgments}
The author is very grateful to Dmitry Turaev
 for pointing out the lack of the Baire property
 for the space of real-analytic maps
 and for his suggestion on how to overcome difficulties caused by it.
The author also thanks to anonymous referees for improvements of the article.
The author was partially supported by JSPS grant-in-aid
 for Scientific Research (C) 26400085.

\section{Smooth case}
Let $\Sigma$ be a $C^\infty$ two-dimensional orientable closed manifold
 and $\Omega$ a $C^\infty$ volume on $\Sigma$.
For an elliptic fixed point $p$ of
 a diffeomorphism $f$ in $\Diff_\Omega^\infty(\Sigma)$,
 the Birkhoff normal form of $f$ at $p$ up to the third order coincides with
\begin{equation*}
\begin{pmatrix} x \\ y \end{pmatrix}
 \mapsto
\begin{pmatrix}
 \cos\alpha & -\sin\alpha \\ \sin\alpha &  \quad\cos\alpha
\end{pmatrix}
\begin{pmatrix} x \\ y \end{pmatrix}
 +O((x^2+y^2)^2),
\end{equation*}
 where $\alpha(x,y)=\alpha_0+\alpha_1(x^2+y^2)$ with real numbers 
 $\alpha_0,\alpha_1$ (see {\it e.g.} \cite[Theorem 2.11]{Mo}).
We say that the elliptic fixed point $p$
 satisfies {\it the Moser stability condition} if $\alpha_1>0$.
The following theorem implies Theorem \ref{thm:main smooth}.
\begin{thm}
\label{thm:smooth periodic}
Suppose that a diffeomorphism $f_0 \in \Diff^\infty_\Omega(\Sigma)$
 admits an elliptic fixed point with the Moser stability condition.
Then, there exists an open neighborhood $\cU$ of $f_0$
 in $\Diff^\infty_\Omega(\Sigma)$
 such that for any finite dimensional closed unit disk $\Delta$
 and any sequence $(\gamma_n)_{n \geq 1}$ of positive integers,
 a generic family $(f_\mu)_{\mu \in \Delta}$ in $\cU$  satisfies that
\begin{equation}
\label{eqn:periodic}
\limsup_{n \ra \infty}
 \frac{1}{\gamma_n} \min\{\#\Per_h(f_\mu,n), \#\Per_e(f_\mu,n)\}= \infty
\end{equation}
 for any $\mu \in \Delta$.
In particular, the condition (\ref{eqn:periodic}) is $C^r$-typical in $\cU$.
\end{thm}

The rest of this section is devoted to
 the proof of Theorem \ref{thm:smooth periodic}.
Fix a finite dimensional closed unit disk $\Delta$.
Let  $(f_\mu)_{\mu \in \Delta}$ be a $C^\infty$ family
 in $\Diff^\infty_\Omega(\Sigma)$
 and $\theta$ a real number.
We call a family $(\psi_\mu)_{\mu \in \Delta}$
  {\it a family of KAM circles} with rotation number $\theta$
 if $(\psi_\mu)_{\mu \in \Delta}$ is a $C^\infty$ family of embeddings
 from $S^1 \times [-1,1]$ to $\Sigma$ such that
\begin{equation*}
 f_\mu \circ \psi_\mu(x,0)=\psi_\mu(x+\theta,0)
\end{equation*}
 and
\begin{equation*}
\left(\frac{\del}{\del y}
 \left[
  P_x \circ (\psi_\mu^{-1} \circ f_\mu \circ \psi_\mu)
  \right]\right) (x,0)>0
\end{equation*}
 for any $\mu \in \Delta$ and $x \in S^1$,
 where $P_x:S^1 \times [-1,1] \ra S^1$ is the projection
 to the first coordinate.
The former condition means that $f_\mu$ is $C^\infty$ conjugate
 to the rigid rotation of angle $\theta$ on
 the invariant circle $\psi_\mu(S^1 \times \{0\})$.
The latter means that $f_\mu$ `twists' a neighborhood of the circle.
The following is an easy consequence of a classical result by Moser
 (see {\it e.g.}, \cite[Theorem 2.11]{Mo}).
\begin{prop}
\label{prop:smooth Moser}
Let $f_0$ be a diffeomorphism in $\Diff_\Omega^\infty(\Sigma)$
 and $p$ be an elliptic fixed point of $f_0$
 with the Moser stability condition.
Then, there exists a neighborhood $\cU$ of $f_0$
 in $\Diff^\infty_\Omega(\Sigma)$ such that
 any $C^\infty$ family $(f_\mu)_{\mu \in \Delta}$ of maps in $\cU$
 admits a family of KAM circles.
\end{prop}

The following proposition implies Theorem \ref{thm:smooth periodic},
 combining with Proposition \ref{prop:smooth Moser}.
\begin{prop}
\label{prop:smooth periodic}
Let $(f_\mu)_{\mu \in \Delta}$ be
 a $C^\infty$ family of maps in $\Diff_\Omega^r(\Sigma)$
 which admits a family of KAM circles.
For any neighborhood $\cV$ of $(f_\mu)_{\mu \in \Delta}$
 in the space of $C^\infty$ families of maps in $\Diff_\Omega^r(\Sigma)$,
 any sequence $(\gamma_n)_{n \geq 1}$ of positive integers,
 and any $n_0 \geq 1$,
 there exists a family $(\hat{f}_\mu)_{\mu \in \Delta}$ in $\cV$
 and an integer $\hat{n} \geq n_0$ such that
\begin{equation*}
\min\{\#\Per_h(\hat{f}_\mu, \hat{n}),
  \#\Per_e(\hat{f}_\mu,\hat{n})\}\geq \hat{n} \cdot \gamma_{\hat{n}}
\end{equation*}
 for any $\mu \in \Delta$.
\end{prop}
\begin{proof}[Proof of Theorem \ref{thm:smooth periodic}
 from Proposition \ref{prop:smooth periodic}]
Let $\cU$ be an open neighborhood of $f_0$ in $\Diff^\infty_\Omega(\Sigma)$ 
 given by Proposition \ref{prop:smooth Moser}.
Fix a sequence $(\gamma_n)_{n \geq 1}$ of positive integers.
Let $\cV$ be the set of $C^\infty$ families
 $(f_\mu)_{\mu \in \Delta}$ of maps in $\cU$ .
Put
\begin{equation*}
 \cV_n=\{(f_\mu)_{\mu \in \Delta} \in \cV \mid
 \min\{\#\Per_h(f_\mu,n), \#\Per_e(f_\mu,n)\}\geq \gamma_n
 \text{ for any }\mu \in \Delta\}.
\end{equation*}
By the persistence of hyperbolic or elliptic fixed points,
 the set $\cV_n$ is open for any $n \geq 1$.
By Proposition \ref{prop:smooth periodic},
 the set $\bigcup_{n \geq n_0} \cV_n$ is dense in $\cV$
 for any $n_0 \geq 1$.
The set $\cR=\bigcap_{n_0 \geq 0}\bigcup_{n \geq n_0}\cV_n$ is
 a residual subset of $\cV$ and
 any family $(f_\mu)_{\mu \in \Delta}$ in $\cR$
 satisfies (\ref{eqn:periodic}).
\end{proof}

In order to prove Proposition \ref{prop:smooth periodic},
 we produce arbitrary many periodic points
 by a series of perturbations on KAM circles.

\begin{lemma}
\label{lemma:smooth rational}
Let  $(f_\mu)_{\mu \in \Delta}$ of maps in $\Diff^\infty_\Omega(\Sigma)$
 and $(\psi_\mu)_{\mu \in \Delta}$ a family of KAM circles.
There exists a family $(\Phi_\mu)_{\mu \in \Delta}$
 of $\Omega$-preserving flows on $\Sigma$ such that
 the map $(p,t,\mu) \mapsto \Phi_\mu^t(p)$ is of class $C^\infty$
 and $\Phi_\mu^t(\psi_\mu(x,0))=\psi_\mu(x+t,0)$
 for any $\mu \in \Delta$, $x \in S^1$, and $t \in \RR$.
\end{lemma}
\begin{proof}
Let $\beta_\mu:S^1 \times [-1,1] \ra \RR$ be the function given by
 $\beta_\mu(x,y) dx\wedge dy=(\psi_\mu)^*\Omega$.
Take a $C^\infty$ function $H$ on $\Sigma \times \Delta$
 such that $H(\psi_\mu(x,y),\mu)=y \cdot \beta_\mu(x,y)$
 on $S^1 \times [-1,1] \times \Delta$
 and put $H_\mu(p)=H(p,\mu)$
 for $p \in \Sigma$ and $\mu \in\Delta$.
Let $\Phi_\mu$ be the Hamiltonian flow
 associated with the Hamiltonian function $H_\mu$ 
 and the symplectic form $\Omega$.
Then, the map $(p,t,\mu) \mapsto \Phi_\mu^t(p)$ is of class $C^\infty$.
The vector field $D\psi_\mu^{-1}(X_\mu)$ on $S^1 \times [-1,1]$
 is the Hamiltonian vector field associated
 with the Hamiltonian function $y \cdot \beta_\mu(x,y)$
 and the symplectic form $\beta_\mu dx \wedge dy$.
This vector field satisfies that
\begin{equation*}
 D\psi_\mu^{-1}(X_\mu)(x,0)= \frac{\del}{\del x}.
\end{equation*}
Hence, we have $\Phi_\mu^t(\psi_\mu(x,0))=\psi_\mu(x+t,0)$.
\end{proof}

\begin{lemma}
\label{lemma:twist}
Let $f:S^1 \times [-1,1] \ra S^1 \times \RR$ be a $C^1$ embedding
 homotopic to the inclusion map 
 and $g:S^1 \ra [-1,1]$ be a $C^1$ map
 such that $\Gamma_g=\{(x,g(x)) \mid x \in S^1\}$ is $f$-invariant
 and $\frac{\del}{\del y}[P_x \circ f]>0$ on $\Gamma_g$,
 where $P_x:S^1 \times [-1,1] \ra S^1$ is the projection
 to the first coordinate.
Then, 
 $\frac{\del}{\del y}[P_x \circ f^n]>0$ on $\Gamma_g$
 for any $n \geq 1$.
\end{lemma}
\begin{proof}
Define a (half) cone $C(x,g(x))$ in $T_{(x,g(x))}(S^1 \times \RR)$
 for $(x,g(x)) \in \Gamma_g$ by
\begin{equation*}
C(x,g(x))
 =\left\{ \left.s\frac{\del}{\del x}+t \frac{\del}{\del y}
 \; \right|\;  s>0,t>g'(x)s\right\}.
\end{equation*}
Since the restriction of $f$ to $\Gamma_g$ is
 an orientation preserving diffeomorphism
 and $(\del/\del y)[P_x \circ f]>0$ on $\Gamma_g$,
 we have $Df_p(\del/\del y) \in C(f(p))$
 and $Df_p(C(p)) \subset C(f(p))$
 for any $p \in \Gamma_g$.
By induction,
 $Df^n_p(\del/\del y) \in C(f^n(p))$ for any $p \in \Gamma_g$ and $n \geq 1$.
This means that $(\del/\del y)[P_x \circ f^n]>0$ on $\Gamma_g$,
\end{proof}

The following lemma is probably well-known for experts.
However, we give a proof because we use an analogous argument
 for the real-analytic case.
\begin{lemma}
\label{lemma:smooth perturb} 
Let  $(f_\mu)_{\mu \in \Delta}$ be a $C^\infty$ family
 of maps in $\Diff_\Omega^r(\Sigma)$,
 $(\psi_\mu)_{\mu \in \Delta}$ a family of KAM circles
 with a rational rotation number $\theta$,
 and $N$ the denominator of $\theta$.
For any neighborhood $\cV$ of the family $(f_\mu)_{\mu \in \Delta}$
 and $\gamma \geq 1$,
 there exists a family $(\hat{f}_\mu)_{\mu \in \Delta}$ in $\cV$
 such that
\begin{equation*}
\min\{\#\Per_h(\hat{f}_\mu,N),
  \#\Per_e(\hat{f}_\mu,N)\}\geq N \gamma
\end{equation*}
 for any $\mu \in \Delta$.
\end{lemma}
\begin{proof}
Since $f_\mu$ preserves the circle $\psi_\mu(S^1 \times \{0\})$,
 we can take $\delta>0$ such that
 the map $F_\mu=\psi_\mu^{-1} \circ f_\mu \circ \psi_\mu$
 is well-defined on $S^1 \times [-\delta,\delta]$.
Remark that $F_\mu(x,0)=(x+\theta,0)$
 and $(\del/\del y)[P_x \circ F_\mu](x,0)>0$ 
 for any $x \in S^1$.
Put $x_{i,j}=\frac{i}{2\gamma N}+j\theta \in S^1$
 and $p_{i,j}=(x_{i,j},0)$ for $i=0,\dots,2\gamma-1$ and $j=0,\dots,N$.
Then, $F_\mu(p_{i,j})=p_{i,j+1}$ and $p_{i,N}=p_{i,0}$.
In particular, $\{p_{i,0},\dots,p_{i,N-1}\}$
 is an $N$-periodic orbit of $F_\mu$.

Take a $C^\infty$ function $h$ on $S^1$
 such that the support of $h$ is contained in $S^1 \times (-\delta,\delta)$,
 $h'(x_{i,j})=0$ and
\begin{equation*}
 h''(x_{i,j}) =
\begin{cases}
 (-1)^i & (j=0) \\ 0 & (j=1,\dots,N-1)
\end{cases}
\end{equation*}
 for any $i=0,\dots,2\gamma-1$ and $j=0,\dots,N-1$.
Put
 $A_\mu=\psi_\mu(S^1 \times (-\delta,\delta))$ for $\mu \in \Delta$
 and define a $C^\infty$ function $H$ on $\Sigma \times \Delta$ by
 $H(\cdot,\mu)=0$ on $\Sigma \setminus A_\mu$
 and
\begin{equation*}
 H(\psi_\mu(x,y),\mu) = -h(x)
\end{equation*}
 for any $(x,y) \in S^1 \times [-\delta/2,\delta/2]$ and any $\mu \in \Delta$.
Let $\Psi_\mu$ be the Hamiltonian flow on $\Sigma$ associated with
 the smooth Hamiltonian function $H(\cdot,\mu)$.
Since the support of the flow $\Psi_\mu$ is contained in $A_\mu$,
 $G_\mu^t(x,y)=\psi_\mu^{-1} \circ \Psi_\mu^t \circ \psi_\mu(x,y)$
 is well-defined 
 for any $\mu \in \Delta$, $(x,y) \in S^1 \times (-\delta,\delta)$,
 and $t \in \RR$.

It is sufficient to show that
\begin{equation*}
\min\{\#\Per_h(F_\mu \circ G_\mu^t),N),
  \#\Per_e((F_\mu \circ G_\mu^t),N)\}\geq N \gamma.
\end{equation*}
 for any sufficiently small $t>0$.
By construction of the Hamiltonian function $H(\cdot,\mu)$,
 we have $ G_\mu^t(x,y)=(x, y+h'(x)t)$ and hence,
\begin{equation*}
 (DG_\mu^t)_{(x,y)}=
\begin{pmatrix} 1 & 0 \\ h''(x)t & 1.
\end{pmatrix}
\end{equation*}
 for any $(x,y) \in S^1 \times (-\delta/2,\delta/2)$
 and any $t$ with $|y+h'(x)t| \leq \delta/2$.
Since $h'(x_{i,j})=0$,
 each $p_{i,j}$ is a fixed point of $G_\mu^t$, and hence,
\begin{align}
\label{eqn:smooth perturb 1}
 (F_\mu \circ G_\mu^t)(p_{i,j}) & =F_{\mu}(p_{i,j})=p_{i,j+1}
\end{align}
 for any $t \geq 0$.
In particular,  $\{p_{i,0},\dots,p_{i,N-1}\}$ is
 an $N$-periodic orbit of $F_\mu \circ G_\mu^t$
 for each $i=0,\dots,2\gamma-1$.
By the condition on $h''$, we also have
\begin{align}
\label{eqn:smooth perturb 2}
D(F_\mu \circ G_\mu^t)_{p_{i,j}}
 & =
\begin{cases}
  (DF_\mu^{N})_{p_{i,0}}
 \begin{pmatrix} 1 & 0 \\ (-1)^i t & 1 \end{pmatrix}
 & (j=0)\\
  (DF_\mu^{N})_{p_{i,j}}
 & (j=1,\dots,N-1).
\end{cases}
\end{align}
To determine the hyperbolicity or ellipticity of
 the periodic orbit $\{p_{i,0},\dots,p_{i,N-1}\}$,
 it is sufficient to
 compute the trace of $D(F_\mu \circ G_\mu^t)^{N}_{p_{i,0}}$.
By Equations (\ref{eqn:smooth perturb 1}) and (\ref{eqn:smooth perturb 2}),
\begin{align}
\tr D(F_\mu \circ G_\mu^t)^{N}_{p_{i,0}}
 & =\tr\left[
 D(F_\mu \circ G_\mu^t)_{p_{i,N-1}}
 \cdots D(F_\mu \circ G_\mu^t)_{p_{i,0}}
  \right] \nonumber\\
 & = \tr\left[
 (DF_\mu^{N})_{p_{i,0}}
 \begin{pmatrix} 1 & 0 \\ (-1)^i t & 1 \end{pmatrix}
\right] \nonumber\\
\label{eqn:smooth trace}
 & = \tr (DF_\mu^{N})_{p_{i,0}}
 + (-1)^i t \cdot \frac{\del}{\del y}[P_x \circ F_\mu^{N}](p_{i,0}).
\end{align}
Since $F_\mu^{N}(x,0)=(x+N\theta,0)=(x,0)$ for any $x \in S^1$,
 the derivative $(DF_\mu^N)_{(x,0)}$ has the form
\begin{equation*}
\begin{pmatrix} 1 & * \\  0 & * \end{pmatrix}.
\end{equation*}
Since $F_\mu$ preserves a volume form,
 we have $\det (DF_\mu^{N})_{(x,0)}=1$, and hence,
 $\tr(DF_\mu^{N})_{(x,0)}=1$.
By Lemma \ref{lemma:twist},
 we also have $\frac{\del}{\del y}[P_x \circ F_\mu^{N}](x,0)>0$.
Hence, by Equation (\ref{eqn:smooth trace}),
 the $N$-periodic orbit $\{p_{i,0},\dots,p_{i,N-1}\}$
 of $(F_\mu \circ G_\mu^t)$ is hyperbolic if $j=0,\dots,2\gamma-1$ is even 
 and elliptic if $j$ is odd for any sufficiently small $t>0$.
This implies that
\begin{equation*}
\min\{\#\Per_h((F_\mu \circ G_\mu^t),N),
  \#\Per_e((F_\mu \circ G_\mu^t),N)\}\geq N \gamma.
\end{equation*}
For any sufficiently small $t>0$,
 the family $(F_\mu \circ G_\mu^t)_{\mu \in \Delta}$ is contained in $\cV$.
\end{proof}

Now, we prove Proposition \ref{prop:smooth periodic}.
Fix an open neighborhood $\cV$ of the family $(f_\mu)_{\mu \in \Delta}$,
 a sequence $(\gamma_n)_{n \geq 1}$ of positive integers,
 and $n_0 \geq 1$.
By Lemma \ref{lemma:smooth rational},
 there exist a rational number $\theta$
 and a family $(\bar{f}_\mu)_{\mu \in \Delta}$ in $\cV$
 such that the denominator $N$ of $\theta$ is greater than $n_0$
 and the family $(\bar{f}_\mu)_{\mu \in \Delta}$
 admits a KAM circle with rotation number $\theta$.
By Lemma \ref{lemma:smooth perturb},
 there exists a family $(\hat{f}_\mu)_{\mu \in \Delta}$ in $\cV$
 such that
\begin{equation*}
 \min\{\#\Per_h(\hat{f}_\mu,N),\#\Per_e(\hat{f}_\mu,N)\}
  \geq N\cdot \gamma_N.
\end{equation*}
 for any $\mu \in \Delta$.

\section{Real analytic case}
\label{sec:analytic}

\subsection{The $C^\omega$ topology}
\label{subsec:topology}
In this subsection, we describe the $C^\omega$-topology
 of the set of real-analytic Hamiltonian diffeomorphisms of $\TT^2$
 and show maps whose all $n$ periodic points are non-degenerate
 are open and dense with respect to the topology.

Let $M_1$ and $M_2$ be complex manifolds
 and $d_{M_2}$ a distance on $M_2$ compatible with the topology of $M_2$.
By $\cH(U,M_2)$, we denote the set of holomorphic maps from $M_1$ to $M_2$.
For a compact subset $K$ of $M_1$, we define a pseudo-distance $d_K$
 on $\cH(M_1,M_2)$ by
\begin{equation*}
 d_K(f,g)=\sup_{z \in K}d_{M_2}(f(z),g(z)).
\end{equation*}
The space $\cH(M_1,M_2)$ is endowed with the topology 
 defined by the family of pseudo-distances $d_K$,
 where $K$ runs over all compact subsets of $M_1$.

Recall that $\Ham^\omega(\TT^2)$ is
 the set of real-analytic Hamiltonian diffeomorphisms on $\TT^2$.
For an open neighborhood $U$ of $\TT^2$ in $\CC^2/\ZZ^2$,
 let $\Ham^\omega_U(\TT^2)$ be the set of maps in $\Ham^\omega(\TT^2)$
 which extend to holomorphic maps from $U$ to $\CC^2/\ZZ^2$.
Remark that the holomorphic extension to $U$ is unique
 for each map in $\Ham^\omega_U(\TT^2)$.
In the rest of this article,
 we identify a map $f$ in $\Ham^\omega_U(\TT^2)$
 and its holomorphic extension to $U$.
The space $\Ham^\omega(\TT^2)$
 is the union of $\Ham^\omega_U(\TT^2)$'s, where $U$ runs over
 all open neighborhoods of $\TT^2$ in $\CC^2/\ZZ^2$.
Each $\Ham^\omega_U(\TT^2)$ is endowed with
 the topology as a subset of $\cH(U,\CC^2/\ZZ^2)$.
{\it The $C^\omega$-topology} of $\Ham^\omega(\TT^2)$ is
 the direct limit topology associated with the inclusions
 $\Ham^\omega_U(\TT^2) \hookrightarrow \Ham^\omega(\TT^2)$.
More precisely,
 a subset $\cU$ of $\Ham^\omega(\TT^2)$ is open if and only if
 $\cU \cap \Ham^\omega_U(\TT^2)$ is an open subset of $\Ham^\omega_U(\TT^2)$
 for any open neighborhood $U$ of $\TT^2$ in $\CC^2/\ZZ^2$.
Remark that this topology is weaker than another topology
 on the space of  real-analytic maps defined by Takens \cite{Ta}
 (see also \cite{BT}).
Takens' topology has the Baire property, but our topology does not.
\footnote{
Takens' topology is too strong for our purpose.
For example,
 let $\Phi=(\Phi^t)_{t \in \RR}$ be a real-analytic flow on $\TT^2$
 generated by a real-analytic vector field $X$.
If $X$ does not extends to $\CC^2/\ZZ^2$, then
 the map  $t \mapsto \Phi^t \in \Diff^\omega(\TT^2)$
 is not  continuous with respect to Takens' $C^\infty$ topology
 while it is continuous with respect to our $C^\omega$ topology.
}

For $n \geq 1$,
 let $\ND_n$ be the subset of $\Ham^\omega(\TT^2)$
 consisting of maps whose
 all periodic points of period less than $n+1$ are non-degenerate.
\begin{prop}
\label{prop:omega non-degenerate}
Let $U$ be an openneighborhood of $\TT^2$ in $\CC^2/\ZZ^2$
 whose closure is compact.
Then, the subset $\ND_n \cap \Ham^\omega_U(\TT^2)$
 is open and dense in $\Ham^\omega_U(\TT^2)$
 for any $n \geq 1$.
\end{prop}
\begin{proof}
By the $C^1$ persistence of non-degenerate periodic points,
 $\ND_n \cap \Ham^\omega_U(\TT^2)$ is an open subset of
 $\Ham^\omega_U(\TT^2)$.

To show density, we follow the argument in \cite[Section 11.3]{Ro}
 (see also \cite{BT}).
Let $Y_1,\dots,Y_4$ be the Hamiltonian vector fields associated to
 Hamiltonian functions $\cos(2\pi x)$, $\sin(2\pi x)$,
 $\cos(2\pi y)$ and $\sin(2\pi y)$.
Then, for any $(x,y) \in \TT^2$,
 the tangent space  $T_{(x,y)}\TT^2$
 is spanned by $\{Y_1(x,y), \dots, Y_4(x,y)\}$.
Let $\Psi_i$ be the Hamiltonian flow generated by the vector field $Y_i$.
Fix $f_0 \in \Ham^\omega_U(\TT^2)$
 and define a $C^\omega$ map
 $F_n:\TT^2 \times \RR^4 \ra \TT^2 \times \TT^2$ by
 $F_n(t_1,\dots,t_4,(x,y))=
 ((x,y), (\Psi_1^{t_1} \circ \dots \circ \Psi_4^{t_4} \circ f_0)^n(x,y))$.
Since $Y_i$ extends to $\CC^2/\ZZ^2$
 and the closure of $U$ is compact,
 $\Psi_i^t$ extends to $U$ for any sufficiently small $t>0$.
By the choice of $Y_1,\dots, Y_4$,
 the map $F_n$ is transverse to the diagonal
 $\{(p,p) \mid p \in \TT^2\}$ in $\TT^2 \times \TT^2$
 if $f$ belongs to $\ND_{n-1}$.
Using this map $F_n$,
 we can show that the density of $\ND_n$
 in $\Ham^\omega_U(\TT^2)$
 by the same argument as in \cite[Section 11.3]{Ro}.
\end{proof}

\subsection{Persistence of KAM curves}
\label{subsec:KAM}
Put $A(r)=\{z \in \CC/\ZZ \mid |\Img z|<r\}$ for $r >0$.
By $P_x$ and $P_y$,
 we denote the natural projections from $\TT^2$
 to the first and second coordinates.
For an open neighborhood $U$ of $\TT^2$ in $\CC^2/\ZZ^2$,
 a map $f \in \Ham^\omega_U(\TT^2)$, and a real number $\theta$,
 we call a holomorphic map $\vphi:A(r) \ra \CC^2/\ZZ^2$
 a {\it graph KAM curve} with rotation number $\theta$ if 
 $f \circ \vphi(z)=\vphi(z+\theta)$ for any $z \in A(r)$,
 $P_x \circ \vphi$ is a diffeomorphism between $A(r)$ onto its image,
 and the twist condition
\begin{equation*}
 \frac{\del}{\del y} (P_x \circ f)>0
\end{equation*}
 holds on $\vphi(S^1)$.
For a graph KAM curve $\vphi$ and $0<r' \leq r$, we put
\begin{equation*}
 U_{\vphi,r'}=\{(z_1,z_2) \in U \mid z_1 \in P_x \circ \vphi(A(r'))\}.
\end{equation*}

We say that a real number $\theta$ is {\it Diophantine} if
 there exists $\tau>0$ and $c>0$ such that
\begin{equation*}
 \left|\theta -\frac{p}{q}\right|\geq \frac{c}{|p|^{2+\tau}}
\end{equation*}
 for any rational number $p/q$.
We say that
 a map $f:S^1 \times (a,b) \ra S^1 \times \RR$ has
 {\it the intersection property} if for any continuous map
 $g:S^1 \ra (a,b)$ the curve $C_g=\{(x,g(x)) \mid x \in S^1\}$
 intersects with $f(C_g)$.
It is known that if $f_H$ is a Hamiltonian diffeomorphism of $\TT^2$
 and $\vphi:S^1 \times (a,b) \ra \TT^2$ be an embedding, then
 $\vphi^{-1} \circ f_H \circ \vphi$ satisfies the intersection property.
For $r>0$ and $\delta>0$, put
\begin{equation*}
 V(r,\delta)=\{(z_1,z_2) \in (\CC/\ZZ) \times \CC \mid 
 |\Img z_1|<r, |z_2|<\delta\}.
\end{equation*}
The following is a version of the KAM theorem
 proved in \cite[Section 32, p.230--235]{SM}.
\begin{thm}
\label{thm:KAM}
For any given Diophantine real number $\theta$,
 real numbers $\alpha>0$, $0<\hat{r}<r$, and $\eta>0$,
 there exists $0<\delta_*<1$ such that
 for any holomorphic map $F:V(r,\delta) \ra (\CC/\ZZ) \times \CC$
 and $0<\delta<\delta_*$ satisfying that
\begin{enumerate}
\item $F(S^1  \times (-\delta,\delta)) \subset S^1 \times \RR$, 
\item the restriction of $F$ on $S^1 \times (-\delta,\delta)$ has
 the intersection property, and
\item $\sup_{(z_1,z_2) \in V(r,\delta)}
 \|F(z_1,z_2)-(z_1+\theta+\alpha z_2,z_2)\|< \delta^{\frac{3}{2}}$,
\end{enumerate}
 we can find a holomorphic map
 $\psi:A(\hat{r}) \ra V(r,\delta)$
 such that $F \circ \psi(z)=\psi(z+\theta)$ for any $z \in A(\hat{r})$
 and $\sup_{z \in A(\hat{r})}\|\psi(z)-(z,0)\| <\eta$.\footnote{
In  \cite[Section 32, p.230--235]{SM},
 the proof is given for the case $\hat{r}=r/2$.
But, it is easy to show the theorem in our form
 by a modification of constants.
The constants $\delta$ and $\delta^{\frac{3}{2}}$ in the theorem correspond
 to $s_0$ and $d_0$ in \cite{SM}.}
\end{thm}

The following lemma corresponds to Proposition \ref{prop:smooth Moser}
 in real-analytic case. 
\begin{lemma}
\label{lemma:omega KAM} 
Let $U$ be an open neighborhood of $\TT^2$ in $\CC^2/\ZZ^2$,
 $f$ a diffeomorphism in $\Ham^\omega_U(\TT^2)$,
 $r>\hat{r}$ positive real numbers,
 $\vphi:A(r) \ra U$ a graph KAM curve with
 a Diophantine rotation number $\theta$,
 and $K$ a compact neighborhood of $\TT^2$ in $\CC^2/\ZZ^2$
 such that $K \subset U_{\vphi,\hat{r}}$.
Then, there exists a neighborhood $\cU$ of $f$ in $\Ham^\omega_U(\TT^2)$
 such that any $\hat{f} \in \cU$ admits a graph KAM curve
 $\hat{\vphi}:A(\hat{r}) \ra U$ with rotation number $\theta$
 such that $K \subset U_{\hat{\vphi},\hat{r}}$.
\end{lemma}
\begin{proof}
Define a map $\psi_1:A(r) \times \CC \ra \CC^2/\ZZ^2$ by
\begin{equation*}
\psi_1(z_1,z_2)=\left(P_x \circ \vphi(z_1),
 \frac{1}{(P_x \circ \vphi)'(z_1)}\cdot z_2+P_y \circ \vphi(z_1)\right).
\end{equation*}
Then, $\psi_1(z,0)=\vphi(z)$ 
 and $f \circ \psi_1(z,0)=\psi_1(z+\theta,0)$.
Hence, $\psi_1(A(r') \times \{0\})$ is
 an $f$-invariant subset of $U_{\vphi,r}$ for any $0<r' \leq r$.
Take $r_1 \in (\hat{r},r)$ and  $\delta_1>0$ such that
 $\psi_1(V(r_1,\delta_1)) \subset U_{\vphi,r}$
 and $F_1=\psi_1^{-1} \circ f \circ \psi_1$ is
 well-defined on $V(r_1,\delta_1)$.
Remark that $F_1(z,0)=(z+\theta,0)$ for any $x \in A(r_1)$
 and the restriction of $F_1$ to $S^1 \times (-\delta_1,\delta_1)$
 preserves the standard volume form $dx \wedge dy$.
Put $\alpha(z)=(\del/\del y)[P_x \circ F_1](z,0)$ for $z \in A(r_1)$.
Then,  for any $x \in S^1$, we have $\alpha(x)>0$ and
\begin{equation*}
 (DF_1)_{(x,0)}=\begin{pmatrix} 1 & \alpha(x) \\ 0 & 1\end{pmatrix}.
\end{equation*}
Since $\theta$ is Diophantine, the equation
\begin{equation*}
 \alpha(x)=\beta(x+\theta)-\beta(x)+\bar{\alpha}
\end{equation*}
 has a solution $\beta:S^1 \ra \RR$ and $\alpha_* \in \RR$
 and the function $\beta$ extends to a holomorphic function on $A(r_1)$.
Since $\alpha(x)>0$ for any $x \in S^1$, we have $\alpha_*>0$.

Define a holomorphic embedding
 $\psi_2: A(r_1) \times \CC \ra (\CC/\ZZ) \times \CC$ by
\begin{equation*}
 \psi_2(z_1,z_2)=(z_1 +\beta(z_1) z_2, z_2).
\end{equation*}
Since $\psi_2(z,0)=(z,0)$,
 there exist $r_2 \in (\hat{r},r_1)$ and $\delta_2 \in (0,\delta_1)$
 such that $F_2=\psi_2^{-1} \circ F_1 \circ \psi_2$ is well-defined on
 $V(r_2,\delta_2)$.
The map $F_2$ satisfies that $F_2(x,0)=(x+\theta,0)$ and
\begin{equation*}
 (DF_2)_{(x,0)}=
\begin{pmatrix} 1 & \beta(x+\theta) \\ 0 & 1 \end{pmatrix}^{-1}
\begin{pmatrix} 1 & \alpha(x) \\ 0 & 1 \end{pmatrix}
\begin{pmatrix} 1 & \beta(x) \\ 0 & 1 \end{pmatrix}
 = \begin{pmatrix} 1 & \alpha_* \\ 0 & 1 \end{pmatrix}.
\end{equation*}
This implies that
 there exist $C>0$, $r_3 \in (\hat{r},r_2)$, and $\delta_3 \in (0,\delta_2)$
 such that
\begin{equation}
\label{eqn:oemga KAM 1}
 \sup_{(z_1,z_2) \in V(r_3,\delta)}
  \|F_2(z_1,z_2) - (z_1+\theta+\alpha_* z_2,z_2)\|\leq C\delta^2
\end{equation}
 for any $\delta \in (0,\delta_3)$.
Put $\psi=\psi_1 \circ \psi_2$.
Then, $\psi(V(r_2,\delta_2)) \subset U_{\vphi,r}$,
 $\psi(z,0)=\vphi(z)$, and
 $F_2=\psi^{-1} \circ f \circ \psi$ has the intersection property.
Fix $r_4 \in (\hat{r},r_3)$.
By Cauchy's integral formula, there exists $\eta>0$ such that
 for any holomorphic map $\Phi:A(r_4) \ra (\CC/\ZZ)\times \CC$
 with $\sup_{z \in A(r_4)}\|\Phi(z)-(z,0)\|<\eta$,
 the the restriction of $P_x \circ (\psi \circ \Phi)$ to $A(\hat{r})$
 is a holomorphic diffeomorphism onto its image
 and $K \subset U_{\psi \circ \Phi,\hat{r}}$.

Let $\delta_* \in (0,\delta_3)$ be the constant in Theorem \ref{thm:KAM}
 for $\theta, \alpha_*, r_3>r_4$ and $\eta$.
Take $\hat{\delta} \in (0,\delta_*)$ such that
 $C\hat{\delta}^2<\hat{\delta}^{\frac{3}{2}}/2$.
Since $F_2=\psi^{-1} \circ f \circ \psi$ is well-defined on $V(r_2,\delta_2)$,
 there exists a neighborhood $\cU$ of $f$ in $\Ham^\omega_U(\TT^2)$
 such that for any $\hat{f} \in \cU$,
 $\psi^{-1} \circ \hat{f} \circ \psi$ is well-defined on $V(r_3,\delta_3)$
 and
\begin{equation*}
 \sup_{(z_1,z_2) \in V(r_3,\delta_3)}
 \|\psi^{-1} \circ \hat{f} \circ \psi(z_1,z_2)-F_2(z_1,z_2)\|
 <\frac{\delta^{\frac{3}{2}}}{2}.
\end{equation*}
With (\ref{eqn:oemga KAM 1}),
 we have
\begin{equation*}
 \sup_{(z_1,z_2) \in V(r_3,\delta_3)}
 \|\psi^{-1} \circ \hat{f} \circ \psi(z_1,z_2)
 -(z_1+\theta+\alpha_* z_2,z_2)\|
 <\delta^{\frac{3}{2}}
\end{equation*}
 for any $\hat{f} \in \cU$.
By Theorem \ref{thm:KAM},
 for $\hat{f} \in \cU$
 there exists a holomorphic map
 $\Phi:A(r_4) \ra V(r_3,\delta_3)$
 such that
 $(\psi^{-1} \circ \hat{f} \circ \psi) \circ \Phi(z)=\Phi(z+\theta)$
 and $\sup_{z \in A(r_4)}\|\psi(z)-(z,0)\|<\eta$.
Put $\hat{\vphi}=\psi \circ \Phi$.
Then, $\hat{f} \circ \hat{\vphi}(z)=\hat{\vphi}(z+\theta)$,
 $P_x \circ \hat{\vphi}:A(\hat{r}) \ra \CC/\ZZ$ is a holomorphic diffeomorphism
 onto its image, and $K \subset U_{\hat{\vphi},\hat{r}}$.
This means that $\hat{\vphi}$ is a graph KAM curve of $\hat{f}$
 such that $K \subset U_{\hat{\vphi},\hat{r}}$.
\end{proof}

\subsection{Proof for the real-analytic case}
\label{subsec:analytic}
The following proposition corresponds to
 Proposition \ref{prop:smooth periodic} in real-analytic case.
\begin{prop}
\label{prop:omega periodic}
Let $U_0$ be an open neighborhood of $\TT^2$ in $\CC^2/\ZZ^2$,
 $f_0$ a diffeomorphism in $\Ham^\omega_{U_0}(\TT^2)$,
 $r_0>r'>r_1$ positive real numbers,
 $\vphi_0:A(r_0) \ra U_0$ a graph KAM curve of $f_0$,
 and $K$ a compact subset of $(U_0)_{\vphi_0,r_1}$.
Then,  there exists an open neighborhood $U_1$ of $\TT^2$
 which satisfies the following conditions:
\begin{enumerate}
\item The closure of $U_1$ is compact.
\item $\vphi_0(A(r')) \subset U_1$ and $K \subset (U_1)_{\vphi_0,r_1}$.
\item For any sequence $(\gamma_n)_{n \geq 1}$ of positive integers,
 any $n_0 \geq 1$, and
 any neighborhood $\cU$ of $f_0$ in $\Ham^\omega_{U_1}(\TT^2)$,
 there exist $f_1 \in \cU$ and $n_1> n_0$ such that
\begin{equation*}
 \min\{\#\Per_h(f_1,n_1),\#\Per_e(f_1,n_1)\} \geq n_1 \cdot \gamma_{n_1}.
\end{equation*}
\end{enumerate}
\end{prop}
\begin{proof}
Proof is analogous to Proposition \ref{prop:smooth periodic}.

Since $K \cup \vphi_0(\cl{A(r')})$
 is a compact subset of $(U_0)_{\vphi_0,r_0}$,
 we can take a compact subsets $K'$ and $K_1$ of $(U_0)_{\vphi_0,r_0}$
 such that  $K \cup \vphi_0(\cl{A(r_1)}) \subset \Int K_1$
 and $K_1 \subset \Int K'$.
Let $U_1$ be interior of $K_1$.
Then, the closure of $U_1$ is compact,
 $K \subset (U_1)_{\vphi_0,r_1}$, and
 $\vphi_0(\cl{A(r_1)}) \subset U_1$.
Put $\vphi_x=P_x \circ \vphi_0$ and
 $g=(P_y \circ \vphi_0) \circ \vphi_x^{-1}$.
Define a holomorphic function $H_1$ on $(U_0)_{\vphi_0,r_0}$ by
\begin{equation*}
 H_1(z_1,z_2)=\frac{1}{2\pi}(\vphi_x)'(z_1) \cdot \sin [2\pi(z_2-g(z_1))].
\end{equation*}
Let $X_1$ be the Hamiltonian vector field associated with
 the Hamiltonian function $H_1$
 and $\Phi_1^t$ the (local) flow generated by $X_1$.
The vector field $X_1$ extends to a holomorphic vector field
 on $(U_0)_{\vphi_0,r_0}$, and hence,
 there exists $T_1>0$ such that $\Phi_1^t$ extends to $K'$
 for any $t \in [0,T_1]$.
The vector field $X_1$ satisfies that
\begin{align*}
 X_1(\vphi_0(z)) & = X_1(\vphi_x(z), g(\vphi_x(z))) \\
 & = \vphi_x'(z)
  \left(\frac{\del }{\del x}+g'(\vphi_x(z))\frac{\del}{\del y}\right)\\
 & = (D\vphi_0)_z\left(\frac{\del}{\del z}\right).
\end{align*}
This implies that $\Phi_1^t(\vphi_0(z))=\vphi_0(z+t)$
 for any $z \in A(r_1)$ and $t \in [0,T_1]$.

Fix a sequence $(\gamma_n)_{n \geq 1}$ of positive integers, $n_0 \geq 1$,
 and an open neighborhood $\cU$ of $f_0$ in $\Ham^\omega_{U_1}(\TT^2)$.
Let $\theta_0$ be the rotation number of the KAM curve $\vphi_0$.
Take $t_1 \in (0,T_1)$ such that 
 $\theta_1=\theta_0+t_1$ is a rational number with denominator $n_1> n_0$
 and $f_0 \circ \Phi_1^{t_1}$ is contained in $\cU$.
Put $\hat{f}=f_0 \circ \Phi_1^{t_1}$,
 $\gamma=\gamma_{n_1}$,
 $p_{i,j}=\vphi_0\left(\frac{i}{2\gamma {n_1}}+j\theta_1\right)$,
 and $x_{i,j}=P_x(p_{i,j})$
 for $i=0,\dots,2\gamma-1$ and $j \geq 0$.
Then,
\begin{equation*}
\hat{f}(x_{i,j},g(x_{i,j}))=\hat{f}(p_{i,j})
 =p_{i,j+1} =(x_{i,j+1},g(x_{i,j+1})). 
\end{equation*}

Using trigonometric functions,
 we can construct a real-analytic function $h_2$ on $S^1$
 which extends to $\CC/\ZZ$ and
 satisfies that $h_2'(x_{i,j})=0$, $h_2''(x_{i,0})=(-1)^i$,
 and $h_2''(x_{i,1})=\dots=h_2''(x_{i,n_1-1})=0$
 for any $i=0,\dots,2\gamma-1$ and $j=0,\dots,n_1-1$.
Put $H_2(x,y)=-h_2(x)$ for $(x,y) \in \TT^2$ and
 let $\Phi_2$ be the Hamiltonian flow
 associated with the Hamiltonian function $H_2$.
The flow $\Phi_2$ satisfies that $\Phi_2^t(x,y)=(x,y+h_2'(x)t)$
 for any $(x,y) \in \TT^2$ and $t \geq 0$.
Since $H_2$ extends to $\CC^2/\ZZ^2$,
 there exists $T_2>0$ such that $\Phi_2^t$ extends to $K_2$
 for any $t \in [0,T_2]$.

Since $\hat{f}^{n_1} \circ \vphi_0(x)=\vphi_0(x+{n_1}\theta_1)=\vphi_0(x)$
 for any $x \in S^1$ and
 $\hat{f}$ preserves the standard volume form $dx \wedge dy$ of $\TT^2$,
 we have $ \tr D\hat{f}^{n_1}_{p_{i,j}}=1$.
By the same argument as in Lemma \ref{lemma:smooth perturb},
 we can show that
$\{p_{i,0}, \dots, p_{i,n_1-1}\}$ is a periodic orbit of
 $\hat{f} \circ \Phi_2^t$ which is hyperbolic if $i$ is even
 and elliptic if $i$ is odd for any $t \in (0,T_2)$.
Take a small $T \in (0,T_2)$ such that
 $\hat{f} \circ \Phi_2^T$ is contained in $\cU$.
Then, $f_1=\hat{f} \circ \Phi_2^T$ satisfies that
\begin{equation*}
 \min\{\#\Per_h(f_1,n_1),\#\Per_e(f_1,n_1)\} \geq n_1 \gamma.
\end{equation*}
\end{proof}

With Lemma \ref{lemma:omega KAM} and
 Proposition \ref{prop:omega periodic},
 we show that any Hamiltonian diffeomorphism with a graph KAM curve
 can be approximated by
 another Hamiltonian diffeomorphism which admits
 a graph KAM curve and many periodic points of high period.
Recall that $\ND_n$ is the subset of $\Ham^\omega(\TT^2)$
 consisting of maps such that
 all periodic points of period less than $n+1$ are non-degenerate.
\begin{lemma}
\label{lemma:omega perturb} 
Let $(\gamma_n)_{n \geq 1}$ be  a sequence of positive integers,
 $K$ a compatc subset of $\CC^2/\ZZ^2$,
 and $\theta$ a Diophantine real number.
Take an open neighborhood $U_0$ of $\TT^2$ in $\CC^2/\ZZ^2$,
 a diffeomorphism $f_0$ in $\Ham^\omega_{U_0}(\TT^2)$,
 positive numbers $r_0>r_1$,
 and a graph KAM curve $\vphi_0:A(r_0) \ra U_0$ with rotation number $\theta$
 such that $K \subset (U_0)_{\vphi_0,r_1}$.
For any given $n_0 \geq 1$ and $\epsilon>0$,
 there exists $n_1> n_0$,
 an open neighborhood $U_1$ of $\TT^2$ in $\CC^2/\ZZ^2$,
 a diffeomorphism $f_1$ in $\Ham^\omega_{U_1}(\TT^2) \cap \ND_{n_1}$,
 and $\vphi_1:A(r_1) \ra U_1$ a graph KAM curve with rotation number $\theta$
 such that $K \subset (U_1)_{\vphi_1,r_1}$,
 $d_K(f_1,f_0) <\epsilon$, and
\begin{equation*}
\min\{\#\Per_h(f_1,n_1), \#\Per_e(f_1,n_1)\} \geq n_1 \cdot \gamma_{n_1}.
\end{equation*} 
\end{lemma}
\begin{proof}
Take $r' \in (r_1,r_0)$.
By Proposition \ref{prop:omega periodic},
 we can find an open neighborhood $U_1$ of $\TT^2$
 satisfying the following conditions:
\begin{enumerate}
\item The closure of $U_1$ is compact.
\item $\vphi_0(A(r')) \subset U_1$ and $K \subset (U_1)_{\vphi_0,r_1}$.
\item For any neighborhood $\cU$ of $f_0$ in $\Ham^\omega_{U_1}(\TT^2)$,
 there exists $f_{\cU} \in \Ham^\omega_{U_1}(\TT^2)$
 and $n_\cU >n_0$ such that
\begin{equation*}
 \min\{\#\Per_h(f_\cU,n_\cU), \#\Per_e(f_\cU, n_\cU)\}
  \geq n_\cU \cdot \gamma_{n_\cU}.
\end{equation*}
\end{enumerate}
The set $\ND_n \cap \Ham^\omega_{U_1}(\TT^2)$ is dense in
 $\Ham^\omega_{U_1}(\TT^2)$ by Proposition \ref{prop:omega non-degenerate}.
Since the numbers $\#\Per_h(f,n)$ and $\#\Per_e(f,n)$
 are lower semi-continuous with respect to $f$,
 we may assume that $f_{\cU}$ is contained in $\ND_{n_{\cU}}$
 by approximating $f_{\cU}$ by a map in $\ND_{n_\cU}$ if it is necessary.
Applying  Lemma \ref{lemma:omega KAM}
 to $U_1$, $f_0$, $r_1<r'$, $\theta$, and $\vphi_0:A(r') \ra U_1$,
 we can find a neighborhood $\cU_1$ of $f_0$
 in $\Ham^\omega_{U_1}(\TT^2)$
 such that any $\hat{f} \in \cU_1$ admits a graph KAM curve
 $\vphi_{\hat{f}}:A(r_1) \ra U_1$ with rotation number $\theta$
 such that $K \subset (U_1)_{\vphi_{\hat{f}},r_1}$.
By shrinking $\cU_1$ if it is necessary, we may assume that
 $\cU_1 \subset \{\hat{f} \in \Ham^\omega_{U_1}(\TT^2) 
 \mid d_K(\hat{f},f_0)<\epsilon\}$.
Put $n_1=n_{\cU_1}$, $f_1=f_{\cU_1}$,
 and $\vphi_1=\vphi_{f_1}$.
Then, the triple $(n_1, f_1,\vphi_1)$ satisfies the required conditions.
\end{proof}

Next, we show that
 any Hamiltonian diffeomorphism with a graph KAM curve
 can be approximated by another one
 which exhibits fast growth of the number of periodic points.
Put $\ND_\infty=\bigcap_{n \geq 1} \ND_n$.
\begin{prop}
\label{prop:omega perturb}
Let $U_0$ be an open neighborhood of $\TT^2$ in $\CC^2/\ZZ^2$,
 $f_0$ a diffeomorphism in $\Ham^\omega_{U_0}(\TT^2)$,
 $\vphi:A(r_0) \ra U_0$ a graph KAM curve with
 a Diophantine rotation number,
 $K$ a compact subset of $(U_0)_{\vphi_0,r_0}$,
 and $(\gamma_n)_{n \geq 1}$ a sequence of positive integers.
Then, we can find an open subset $U_\infty$ of $U_0$ such that
\begin{enumerate}
\item $K \subset U_\infty$ and
\item for any neighborhood $\cU$ of $f_0$ in $\Ham^\omega_{U_\infty}(\TT^2)$,
 there exists $f_\infty \in \cU \cap \ND_\infty$ such that
\begin{equation}
\label{eqn:analytic fast}
\limsup_{n \ra \infty}
 \frac{1}{\gamma_n}\min\{\#\Per_h(f_\infty,n),\#\Per_e(f_\infty,n)\}=\infty.
\end{equation}
\end{enumerate}
\end{prop}
\begin{proof}
Take a compact subset $K_\infty$ of $(U_0)_{\vphi_0,r_0}$ such that
 $K \subset \Int K_\infty$.
Put $U_\infty=\Int K_\infty$
 and fix a neighborhood $\cU$ of $f_0$ in $\Ham^\omega_{U_\infty}(\TT^2)$.
Take $\epsilon_0>0$ and
 a compact set $K'_\infty \subset U_\infty$ such that the set
\begin{equation*}
 \{f \in \Ham^\omega_{U_\infty}(\TT^2) \mid
 d_{K'_\infty}(f,f_0)<\epsilon_0\}
\end{equation*}
 is contaied in $\cU$.
Put $n_0=1$.
We claim that there exist a sequence
 $(U_k,f_k,r_k,\vphi_k,n_k,\epsilon_k)_{k \geq 1}$ 
 of sextuplets such that
\begin{enumerate}
 \item $U_k$ is an open neighborhood of $\TT^2$ in $\CC^2/\ZZ^2$,
 $f_k \in \Ham^\omega_{U_k}(\TT^2)$,
 \item $r_k$ is a positive number and
 $\vphi_k:A(r_k) \ra U_k$ is a graph KAM curve of $f_k$
 with a Diophantine rotation number such that
 $K_\infty \subset (U_k)_{\vphi_k,r_k}$,
 \item $n_k > n_{k-1}$ and $0<\epsilon_k<\epsilon_{k-1}/3$,
 \item $d_{K_\infty}(f_k,f_{k-1})\leq \epsilon_{k-1}/3$, and
 \item any $f \in \Ham^\omega_{U_k}(\TT^2)$
 with $d_{K_\infty}(f,f_k)<\epsilon_k$ is
 contained in $\ND_{n_k}$ and satisfies that
\begin{equation*}
 \min\{\#\Per_h(f,n_k),\#\Per_e(f,n_k)\} \geq n_k \cdot \gamma_{n_k}.
\end{equation*}
\end{enumerate}
 for any $k\geq 1$.
By assumption,
 the quadruplet $(U_0,f_0,r_0,\vphi_0)$ satisfies the first and second items
 in the above claim for $k=0$.
We show that if $(U_{k-1},f_{k-1},r_{k-1},\vphi_{k-1})$ satisfies
 the first and second items in the above claim for some $k \geq 1$.
 then there exists $(U_k,f_k,r_k,\vphi_k,n_k,\epsilon_k)$
 which satisfies all the items.
Then, the claim will be proved by induction.

Take $r_k \in (0,r_{k-1})$ such that
 $K \subset (U_{k-1})_{\vphi_{k-1},r_k}$.
By Lemma \ref{lemma:omega perturb},
 there exist an open neighborhood $U_k$ of $\TT^2$ in $\CC^2/\ZZ^2$,
 $f_k \in \Ham^\omega_{U_k}(\TT^2) \cap \ND_{n_k}$,
 a graph KAM curve  $\vphi_k:A(r_k) \ra U_k$  of $f_k$
 with a Diophantine rotation number such that
 $K_\infty \subset (U_k)_{\vphi_k,r_k}$,
 $d_{K_\infty}(f_k, f_{k-1})<\epsilon_{k-1}/3$,
 and
\begin{equation*} 
 \min\{\#\Per_h(f_k,n_k),\#\Per_e(f_k,n_k)\}
 \geq n_k \cdot \gamma_{n_k}.
\end{equation*}
Since $\ND_{n_k}$ is an open subset of $\Ham^\omega(\TT^2)$,
 the set $\Ham^\omega_{\Int K_\infty}(\TT^2) \cap \ND_{n_k}$
 is an open subset of $\Ham^\omega_{\Int K_\infty}(\TT^2)$.
In $\ND_{n_k}$, the numbers of hyperbolic ({\it resp.} elliptic)
 periodic points of period $n_k$ is locally constant.
Hence, we can take $\epsilon_k \in (0,\epsilon_{k-1}/3)$ such that
 any $f \in \Ham^\omega_{U_k}(\TT^2)$ with $d_{K_\infty}(f,f_k)<\epsilon_k$
 is contained in $\ND_{n_k}$ and satisfies that
\begin{equation*} 
 \min\{\#\Per_h(f,n_k),\#\Per_e(f,n_k)\}
 \geq n_k \cdot \gamma_{n_k}.
\end{equation*}
Therefore, $(U_k,f_k,r_k,\vphi_k,n_k,\epsilon_k)$ satisfies
 all the items in the claim for $k$.

Since $d_{K_\infty}(f_k,f_{k-1})<\epsilon_{k-1}/3$
 and $0<\epsilon_k < \epsilon_{k-1}/3$ for any $k \geq 1$,
 we have $\epsilon_k <3^{-k} \epsilon_0$
 and $d_{K_\infty}(f_l,f_k)<2\epsilon_k/3$ for any $l \geq k \geq 0$.
This implies that the limit $f_\infty=\lim_{k \ra \infty} f_k$ exists
 and it is a continuous map on $K_\infty$ such that
 $d_{K_\infty}(f_\infty,f_k)<\epsilon_k$ for any $k \geq 0$.
Since $d_{K_\infty}(f_\infty,f_0)<\epsilon_0$,
 the map $f_\infty$ is an element of $\cU$.
By the choice of $\epsilon_k$, the map $f_\infty$ is contained in
 $\ND_\infty$ and
\begin{equation*} 
 \min\{\#\Per_h(f_\infty,n_k),\#\Per_e(f_\infty,n_k)\}
 \geq n_k \cdot \gamma_{n_k}
\end{equation*}
 for any $k \geq 1$.
\end{proof}

Now, we prove Theorem \ref{thm:main analytic}.
Let $\cU_*$ be the subset of $\Ham^\omega(\TT^2)$
 consisting of diffeomorphisms admitting a graph KAM curve
 with a Diophantine rotation number.
It is easy to see that the map $\bar{f}(x,y)=(x+\sin(2\pi y),y)$ is 
 contained in $\cU_*$, and hence, $\cU_*$  is non-empty.
By Lemma \ref{lemma:omega KAM},
 for any open neighborhood $U$ of $\TT^2$
 and any $f \in \Ham^\omega_U(\TT^2)$,
 there exists a neighborhood $\cU_f$ of $f$ in $\Ham^\omega_U(\TT^2)$
 such that any $\hat{f} \in \cU_f$ admits a graph KAM curve.
This implies that $\cU_* \cap \Ham^\omega_U(\TT^2)$ is an open subset
 of $\Ham^\omega_U(\TT^2)$ for any $U$.
Hence, $\cU_*$ is an open subset of $\Ham^\omega(\TT^2)$.

Fix a sequence $(\gamma_n)_{n \geq 1}$ of positive integers,
 $f \in \cU_*$, and an open neighborhood $\cU  \subset \cU_*$ of $f$.
In order to prove Theorem \ref{thm:main analytic},
 it is sufficient to show that there exists $f_\infty \in \cU \cap \ND_\infty$
 such that
\begin{equation}
\label{eqn:omega growth}
 \limsup_{n \ra \infty}
 \frac{1}{\gamma_n} \min\{\#\Per_h(f_\infty,n), \#\Per_e(f_\infty,n)\}
 =\infty.
\end{equation}
Let $U$ be an open neighborhood of $\TT^2$ such that
 $f$ is an element of $\Ham^\omega_U(\TT^2)$.
The map $f$ admits a graph KAM curve $\vphi:A(r) \ra U$
 with a Diophantine rotation number.
By Proposition \ref{prop:omega perturb},
 there exist an open subset $U_\infty$ of $U$
 and $f_\infty \in \Ham^\omega_{U_\infty}(\TT^2) \cap \cU \cap \ND_\infty$
 which satisfies the condition (\ref{eqn:omega growth}).

\appendix
{
\section{Another example
 of smooth unimodal map with fast growth}
In this appendix, we give a simple example of
 a $C^\infty$ unimodal map on a compact interval
 which exhibits fast growth of the number of periodic points.
Let $\Per(f,n)$ be the set of periodic point of a map $f$
 whose minimal period is $n$.
\begin{thm}
Let $(\gamma_n)_{n \geq 1}$ be a sequence of positive integers.
Then, there exists a $C^\infty$ unimodal self-map $f$ on $[-1,1]$ such that
 all periodic points of $f$ are hyperbolic and
\begin{equation}
\label{eqn:one-dim}
 \#\Per(f,n) \geq \gamma_n
\end{equation}
 for any $n \geq 2$.
\end{thm}
The unique critical point of the above map must be flat
 as mentioned in Section \ref{sec:introduction}.
For any finite $r \geq 1$,
 Kaloshin and Kozlovski \cite{KK} constructed
 a $C^r$ unimodal map $f$ which is of class $C^\infty$
 except at the critical point
 and satisfies $\#\Per(f,3^k) \geq \gamma_{3^k}$ for any $k \geq 1$.
While their construction of the example starts
 from an infinitely many renormalizable map,
 our construction starts from a Misiurewicz map.
\begin{proof}
Our proof is inspired by
 de Melo's construction of a $C^\infty$ interval map
 with infinitely many periodic attractors \cite{dM}.

Recall that the $C^\infty$ topology of
 the space of $C^\infty$ self-maps on $[-1,1]$
 is given by the metric
\begin{equation*}
 d_{C^\infty}(f,g)=\sum_{s=0}^\infty2^{-s}
 \frac{\|f^{(s)}-g^{(s)}\|_{\sup}}{1+\|f^{(s)}-g^{(s)}\|_{\sup}},
\end{equation*}
 where $\|h\|_{\sup}=\sup_{x \in [-1,1]}|h(x)|$
 and $f^{(s)}$ is the $s$-th derivative of $f$.

Take a $C^\infty$ non-decreasing function $\chi$ on $\RR$ such that
 $\chi(t)=0$ if $t \leq 0$ and $\chi(t)=1$ if $t \geq 1$.
Fix $0<\delta<1/4$ and define a $C^\infty$ function $g_k$ by
\begin{equation*}
g_k(x)=
2^{-k}\left(1+\chi\left((2k-1)(2k\delta^{-1}x-1)\right)\right)
\cdot (1+x).
\end{equation*}
Put $x_k=\delta/(2k+1)$ and $x'_k=\delta/(2k)$.
Then, we have $g_k(x)=g_{k+1}(x)=2^{-k}(1+x)$
 for $x \in [x_k,x'_k]$
 and $\lim_{k \ra \infty} d_{C^\infty}(g_k,0)=0$.
There exists a $C^\infty$
 self-map $f$ on $[-1,1]$ such that
 $x=0$ is the unique critical point of $f_0$ and
\begin{equation*}
 f_0(x)=
\begin{cases}
 1+2x & (x \in \left[-1,-1/3\right]),\\
 1-2x & (x \in \left[1/3,1\right]),\\
 1-g_k(x) & (x \in [x_k,x_{k-1}],\, k \geq 1).
\end{cases}
\end{equation*}
Since $f_0(y-1)=2y-1$ for $y \in [0,2/3]$,
 we have
\begin{align*}
f_0^{k+1}(x)
 & = f_0^{k}(1-2^{-k}(1+x))
  = f_0^{k-1}(2^{-(k-1)}(1+x)-1) 
  = x
\end{align*}
 for any $k \geq 1$ and $x \in [x_k,x'_k]$.

Since $x=0$ is the unique critical point of $f_0$,
 we can take a continuous function $\eta$ on $[0,1]$ such that
 $0<\eta(|x|)<|f_0'(x)|$ for any $x \in [-1,1] \setminus \{0\}$.
Let $C^\infty_\eta([-1,1])$ be the set of $C^\infty$ self-maps $f$
 on $[-1,1]$ such that $0<\eta(|x|)<|f'(x)|$
 for any $x \in [-1,1] \setminus \{0\}$.
For $n \geq 1$, let $\ND'_n$ be the subset of $C^\infty_\eta([-1,1])$
 consisting of map
 whose all periodic points of period less than $n+1$ are hyperbolic.
Remark that $\ND'_n$ is open and dense in $C^\infty_\eta([-1,1])$
 with respect to the $C^\infty$ topology.

For $m \geq 1$, define a subset $I_m$ of $[-1,1]$ by
\begin{align*}
 I_m & = \bigcup_{k=m}^\infty \bigcup_{n=0}^{k} f_0^n([x_k,x'_k])
\end{align*}
Let $\bar{I}_m$ be the closure of $I_m$.
For $0 \leq j \leq k-2$, we have
\begin{equation*}
f_0^{k-j}([x_k,x_{k+1}])
 =[2^{-(j+1)}(1+x_k)-1,2^{-(j+1)}(1+x'_k)-1]
 \subset  [2^{-(j+1)}-1,2^{-j}-1].
\end{equation*}
This implies that
\begin{equation*}
 \bar{I}_m \setminus I_m=\{-1,0,1\} \cup \{2^{-n}-1 \mid n \geq 1\}.
\end{equation*}
Therefore, $x=-1$ is the unique periodic point of $f_0$
 in $\bar{I}_m$ whose period is less than $m+1$.
Fix $\epsilon_0>0$.
By a small perturbation on $[-1,1] \setminus \bar{I}_1$,
 we can take $\hat{f}_0 \in C^\infty_\eta([0,1]) \cap  \ND'_1$
 such that $d_{C^\infty}(\hat{f}_0,f_0)<\epsilon_0/3$ and
 $\supp(\hat{f}_0-f_0) \cap \bar{I}_1=\emptyset$.
By another small perturbation on $[x_1,x_1']$,
 we also obtain $f_1 \in C^\infty_\eta([0,1]) \cap \ND'_1$
 such that $d_{C^\infty}(f_1, f_0)<\epsilon_0/3$,
 $\supp(f_1-f_0) \cap \bar{I}_2=\emptyset$,
 and $\Per(f_1,2) \cap (x_1,x'_1)$
 consists of at least $\gamma_2$ hyperbolic periodic points of period $2$.
Take $\epsilon_1 \in (0,\epsilon_0/3)$ such that
 any $f \in C^\infty_\eta([-1,1])$ with $d_{C^\infty}(f,f_1)<\epsilon_1$
 is contained in $\ND'_1$ and satisfies that $\#\Per(f,2) \geq \gamma_2$.
 
By inductive perturbations in the same way, 
we can obtain sequences $(f_k)_{k \geq 1}$
 of maps in $C^\infty_\eta([-1,1])$
 and $(\epsilon_k)_{k \geq 1}$ of positive real numbers 
 which satisfy the following conditions for any $k \geq 1$:
\begin{enumerate}
 \item $\epsilon_k \in (0,\epsilon_{k-1}/3)$.
 \item $d_{C^\infty}(f_k,f_{k-1})<\epsilon_{k-1}/3$.
 \item $\supp(f_k-f_{k-1}) \cap \bar{I}_{k+1}=\emptyset$.
 \item Any $f \in C^\infty_\eta([-1,1])$
 with $d_{C^\infty}(f,f_k)<\epsilon_k$ is contained in $\ND'_k$
 and satisfies that $\#\Per(f,k+1) \geq \gamma_{k+1}$.
\end{enumerate}
Since $\epsilon_k \in (0,\epsilon_{k-1}/3)$
 and $d_{C^\infty}(f_k,f_{k-1})<\epsilon_{k-1}/3$,
 $f_k$ converges to a $C^\infty$ map $f$ 
 and it satisfies that $d_{C^\infty}(f,f_k)<\epsilon_k$ for any $k \geq 1$.
By the choice of $\epsilon_k$,
 all periodic points of $f$ are hyperbolic
 and $\# \Per(f,k+1) \geq \gamma_{k+1}$ for any $k \geq 1$.
Since $|f_k'(x)|> \eta(x)>0$ for any $x \in [-1,1] \setminus \{0\}$
 and $k \geq 1$,
 we have $|f'(x)| \geq \eta(x)>0$ for any $x \neq 0$.
Therefore, $x=0$ is the unique critical point of $f$.
\end{proof}

}

\end{document}